\documentclass{article}

\usepackage{arxiv}
\usepackage[utf8]{inputenc}
\usepackage[T1]{fontenc}
\usepackage{microtype}
\usepackage{booktabs}
\usepackage{amsmath,amssymb,mathtools,amsthm}
\usepackage{xcolor}
\usepackage{natbib}
\usepackage{float}
\usepackage{hyperref}
\usepackage[capitalize,noabbrev]{cleveref}
\usepackage{enumitem}
\usepackage{tikz}
\usetikzlibrary{arrows.meta,positioning,calc}

\hypersetup{
  hidelinks,
  pdftitle={Infinitesimal Causality},
  pdfsubject={Smooth intervention protocols, Lie brackets, and causal diagnostics},
  pdfauthor={Sridhar Mahadevan},
  pdfkeywords={infinitesimal causality, intervention fields, Lie brackets, statistical manifolds, Markov categories}
}

\theoremstyle{plain}
\newtheorem{theorem}{Theorem}[section]
\newtheorem{proposition}[theorem]{Proposition}
\newtheorem{corollary}[theorem]{Corollary}
\theoremstyle{definition}
\newtheorem{definition}[theorem]{Definition}
\newtheorem{example}[theorem]{Example}

\theoremstyle{remark}

\crefname{theorem}{theorem}{theorems}
\crefname{proposition}{proposition}{propositions}
\crefname{corollary}{corollary}{corollaries}
\crefname{definition}{definition}{definitions}
\crefname{example}{example}{examples}
\crefname{assumption}{assumption}{assumptions}
\crefname{remark}{remark}{remarks}

\newcommand{\R}{\mathbb R}
\newcommand{\M}{\mathcal M}
\newcommand{\Dvis}{\mathcal D_{\mathrm{vis}}}
\newcommand{\Pvis}{\Pi_{\mathrm{vis}}}
\newcommand{\eps}{\varepsilon}
\newcommand{\doop}{\operatorname{do}}
\newcommand{\Span}{\operatorname{span}}

\newcommand{\id}{\mathrm{id}}
\newcommand{\StatInf}{\mathsf{Stat}_{\infty}}
\newcommand{\IC}{\textsc{IC}}

\title{\textbf{Infinitesimal Causality}}

\author{Sridhar Mahadevan\\
Adobe Research and University of Massachusetts Amherst\\
\texttt{smahadev@adobe.com, mahadeva@umass.edu}}

\date{\today}

\begin{document}
\maketitle

\begin{abstract}
Interventions can be varied continuously in many causal models.  Differentiating
a specified smooth intervention protocol produces vector fields on a statistical
model, and their Lie brackets describe the noncommutativity of the corresponding
local perturbations.  We formulate this differential geometry of interventions
on smooth statistical models and call the resulting framework
\emph{infinitesimal causality} (\IC).  Given a constant-rank distribution
spanned by visible intervention fields, we define the normal Lie-bracket
residual and show that its vanishing is exactly the involutivity condition in
the classical Frobenius theorem.  We establish the coordinate invariance of the
zero-residual property and characterize its dependence on the intervention
protocol, visible span, and metric.  Fully observed and latent-variable
examples delineate the additional structural assumptions needed to interpret
bracket residuals causally.

We also distinguish tangent vectors on a statistical parameter manifold from
derivatives of stochastic kernels.  In the finite-state linearization of a
Markov category, normalization and copy compatibility yield well-typed
first-order defects.  Normalization is automatic for differentiable paths of
stochastic kernels, whereas copy compatibility characterizes a more restrictive
deterministic or comonoid-preserving perturbation.  Together, the geometric and
kernel-level constructions make \IC\ a precise foundation for
Lie-bracket-based causal diagnostics and identify the assumptions required to
pass from local intervention geometry to causal conclusions.
\end{abstract}


\section{Introduction}

Causal inference distinguishes passive observation from active intervention
\citep{pearl2009causality,rubin2005causal}.  Pearl's do-calculus is a discrete
identification calculus: its three rules transform interventional probability
expressions under separation conditions in specifically mutilated causal
graphs.  Categorical formulations express the same graph surgery and
conditional-independence reasoning in Markov categories and string diagrams
\citep{fritz2023dseparation,jacobs2018causal}.

A different question arises when the strength, duration, or target of an
intervention varies smoothly.  A one-parameter intervention protocol can be
differentiated at zero strength, producing a tangent direction on a statistical
model.  Applying this construction at every model point gives a vector field.
Several intervention protocols generate a distribution of tangent directions,
and Lie brackets test whether successive infinitesimal perturbations remain in
their visible span.

We call this framework \emph{infinitesimal causality} (\IC).  The acronym refers
to the combined study of smooth intervention fields, their compatibility with
probabilistic information structure, and their closure under Lie brackets.

This geometric construction motivates Lie-bracket diagnostics used in companion
algorithmic work such as BRIDGE/SKFM
\citep{mahadevan2026latentconfoundedcausaldiscovery}.  A Lie bracket depends on
the chosen intervention fields, while a projected residual also depends on the
visible distribution and metric.  Its causal interpretation therefore requires
a model connecting intervention protocols to causal mechanisms.  Noncommuting
interventions can occur in a fully observed smooth system, while commuting
visible interventions can coexist with an omitted common cause; these two cases
provide useful boundary examples for the theory.

The paper makes the following contributions:
\begin{itemize}[leftmargin=*,itemsep=2pt]
  \item the category \(\StatInf\) organizes smooth statistical families with
        specified sufficient-statistic presentations and
        structure-preserving channels;
  \item smooth intervention protocols induce vector fields on a statistical
        manifold;
  \item the normal component of their brackets measures failure of the chosen
        visible distribution to be involutive;
  \item under constant rank, vanishing residuals are equivalent to Frobenius
        integrability;
  \item this zero-residual property is coordinate invariant, although its norm
        and statistical estimation depend on additional choices; and
  \item first-order copy/discard calculations require a separate, well-typed
        linearization of stochastic kernels.
\end{itemize}

The analysis proceeds at two related but distinct levels.  On the statistical
manifold, intervention fields and their brackets determine a local geometric
structure.  In a finite-state Markov setting, derivatives of stochastic kernels
admit normalization and copy-compatibility tests.  A model-specific realization
is needed to relate these levels.  We make that interface explicit and state the
further completeness, exclusion, faithfulness, and estimation assumptions
needed for causal identification.

\section{Frobenius Markov Categories and Tangent Semantics}
\label{sec:categorical-background}

This section recalls the categorical and differential structures used in the
paper.  The two structures enter at different levels.  Markov categories
describe stochastic maps and the classical operations of copying and
discarding data.  Tangent categories axiomatize differentiation.  In the
present framework, tangent semantics are supplied concretely by smooth
parameter manifolds and by linearization in spaces of kernels; we do not assume
that an arbitrary Markov category itself carries a tangent structure.

\subsection{Markov categories and classical data}

A Markov category is a symmetric monoidal category \(\mathcal C\) in which
every object \(X\) is equipped with copying and discarding morphisms
\[
  \Delta_X:X\longrightarrow X\otimes X,
  \qquad
  !_X:X\longrightarrow I.
\]
These maps form a commutative comonoid and are compatible with the monoidal
product \citep{fritz2020synthetic}.  In the familiar category of measurable
spaces and Markov kernels,
\[
  \Delta_X(x)=(x,x)
\]
copies a realized value, while \(!_X\) maps every value to the monoidal unit.
Discarding is compatible with every stochastic map.  Copying is not: a
stochastic kernel generally generates fresh randomness, so copying one sampled
output differs from sampling twice.  Morphisms that preserve copying are the
deterministic morphisms under standard hypotheses.

This distinction is central to the linearized copy test in
\cref{sec:tests}.  The equation
\[
  \Delta_YK=(K\otimes K)\Delta_X
\]
does not express ordinary stochasticity; it says that \(K\) transports
classical, copyable information without introducing randomness.  Its
derivative is therefore a stringent compatibility condition rather than a
generic law of probabilistic perturbations.

\subsection{What additional Frobenius structure means}

A commutative comonoid \((\Delta,! )\) is not, by itself, a Frobenius algebra.
A Frobenius algebra also has multiplication and unit maps
\[
  \mu:X\otimes X\longrightarrow X,
  \qquad
  \eta:I\longrightarrow X,
\]
with monoid, comonoid, and Frobenius compatibility laws.  Such structure can be
specified for classical observables in concrete categorical models, and it is
useful for diagrammatic treatments of equality, comparison, and
disintegration \citep{cho2017disintegration}.  It is additional data; it does
not follow from the axioms of a Markov category.

Accordingly, the results below use only the canonical copy/discard comonoid
unless multiplication and unit are explicitly supplied.  The phrase
\emph{Frobenius Markov} is best understood here as describing a Markov setting
equipped with this additional classical structure, not as a property enjoyed
by every Markov category.  None of the main integrability results requires the
extra maps \(\mu\) and \(\eta\).

There is also a separate use of the name Frobenius.  A categorical Frobenius
algebra concerns multiplication and copying.  The classical Frobenius theorem
in differential geometry concerns involutive distributions and foliations.
The two notions are historically and mathematically distinct.  This paper
relates copy/discard compatibility to intervention geometry only after a
model-specific realization has connected stochastic kernels with vector fields.

\subsection{Sufficient statistics as copyable summaries}

Let \(p:\Theta\rightsquigarrow X\) be a statistical model represented by a
smoothly parameterized Markov kernel, and let
\[
  t:X\longrightarrow S
\]
be a deterministic sufficient statistic.  In an ordinary dominated model,
sufficiency may be expressed by the Fisher--Neyman factorization
\[
  p_\theta(x)=g_\theta(t(x))h(x).
\]
Categorically, the same idea says that the parameter-dependent information in
the observation factors through the statistic, with the remaining
randomization independent of the parameter once \(t(X)\) is known.

Because \(t\) is deterministic, it respects copying:
\[
  \Delta_S t=(t\otimes t)\Delta_X.
\]
Thus a realized statistic can be retained, copied, and passed to multiple
downstream computations as classical information.  This does not mean that
the stochastic observation itself can be duplicated as two independent draws.
It means that, after the observation has been made, its sufficient summary is
a deterministic value.

For an exponential family
\[
  p_\theta(x)
  =
  h(x)\exp\{\langle\theta,T(x)\rangle-A(\theta)\},
\]
the sufficient statistic also generates the score directions:
\[
  \partial_i\log p_\theta(x)
  =
  T_i(x)-\mathbb E_\theta[T_i(X)].
\]
Sufficient statistics therefore connect the Markov description of classical
data with the tangent description of a regular statistical family.

\subsection{The category
\texorpdfstring{\(\StatInf\)}{Stat-infinity} of presented statistical models}
\label{sec:statinf}

It is useful to collect the preceding data into a category.  The subscript
\(\infty\) indicates that the parameter spaces and parameter dependence are
smooth; it does not assert that this category has already been equipped with a
tangent-category structure.

\begin{definition}[Presented regular statistical model]
\label{def:statinf-object}
An object of \(\StatInf\) is a tuple
\[
  A=(\Theta,X,S,p,t)
\]
consisting of:
\begin{enumerate}[leftmargin=*,itemsep=2pt]
  \item a finite-dimensional smooth parameter manifold \(\Theta\);
  \item measurable sample and statistic spaces \(X\) and \(S\);
  \item a smoothly parameterized Markov kernel
        \(p:\Theta\rightsquigarrow X\), written
        \(\theta\mapsto p_\theta\);
  \item a deterministic sufficient-statistic map \(t:X\to S\); and
  \item a regularity stratum on which the Fisher information is smooth and
        positive definite.
\end{enumerate}
\end{definition}

The sufficiency requirement may be stated by a factorization in dominated
models or by the corresponding conditional-independence factorization in the
ambient Markov category.  Including the map \(t\) makes the chosen classical
summary part of the presentation; two objects may describe the same family
with different sufficient presentations.

\begin{definition}[Morphisms in \(\StatInf\)]
\label{def:statinf-morphism}
Let
\[
  A=(\Theta,X,S,p,t),
  \qquad
  B=(\Phi,Y,S',q,t')
\]
be objects.  A morphism
\[
  (f,K,h):A\longrightarrow B
\]
consists of a smooth map \(f:\Theta\to\Phi\), a smoothly
\(\theta\)-dependent Markov kernel \(K_\theta:X\rightsquigarrow Y\), and a
deterministic map \(h:S\to S'\), subject to
\[
  q_{f(\theta)}=K_\theta\circ p_\theta
  \tag{S1}
\]
as states on \(Y\), and
\[
  t'\circ K_\theta=h\circ t
  \tag{S2}
\]
as kernels \(X\rightsquigarrow S'\).  Equation (S1) says that the target
statistical family is the pushforward of the source family.  Equation (S2)
says that the selected sufficient summary is transported deterministically by
\(h\).
\end{definition}

The second equation is intentionally restrictive.  A general stochastic
channel between sample spaces need not preserve a chosen sufficient statistic.
Such a channel remains a morphism in the ambient Markov category, but it is not
a morphism of these \emph{presented} statistical models unless the chosen
summaries commute with it.

\begin{proposition}
\label{prop:statinf-category}
The objects and morphisms above form a category \(\StatInf\).
\end{proposition}

\begin{proof}
The identity on \(A\) is
\((\id_\Theta,\id_X,\id_S)\).  Given
\[
  (f,K,h):A\to B,
  \qquad
  (g,L,k):B\to C,
\]
define their composite by
\[
  (g,L,k)\circ(f,K,h)
  =
  \bigl(g\circ f,\,
        \theta\mapsto L_{f(\theta)}\circ K_\theta,\,
        k\circ h\bigr).
\]
The pushforward equation follows from
\[
  L_{f(\theta)}K_\theta p_\theta
  =
  L_{f(\theta)}q_{f(\theta)}
  =
  r_{g(f(\theta))}.
\]
For the statistic equation,
\[
  t''L_{f(\theta)}K_\theta
  =
  k\,t'K_\theta
  =
  kh\,t.
\]
Associativity follows from associativity of smooth maps and composition of
Markov kernels; the displayed identities act as left and right units.
\end{proof}

There is a forgetful functor
\[
  U:\StatInf\longrightarrow\mathbf{Smooth},
  \qquad
  U(\Theta,X,S,p,t)=\Theta,
  \qquad
  U(f,K,h)=f.
\]
Composing \(U\) with the ordinary tangent functor gives the parameter-space
tangent bundle \(T\Theta\) and differential \(Tf\).  This is the tangent
semantics used for intervention fields.

It is important that \(T\circ U\) is a functor
\(\StatInf\to\mathbf{Smooth}\), not automatically an endofunctor on
\(\StatInf\).  The manifold \(T\Theta\) does not by itself specify a normalized
statistical model on \(X\), a sufficient statistic, or a lifted kernel for each
morphism.  Constructing those data coherently would be the additional step
needed to make \(\StatInf\) a tangent category.  Similarly, the sample and
statistic spaces live in an ambient Markov category whose copy/discard maps are
available there; this alone does not make \(\StatInf\) itself a Markov category.

\subsection{Tangent categories}

A tangent category is a category \(\mathcal X\) equipped with a tangent functor
\(T:\mathcal X\to\mathcal X\), natural transformations playing the roles of
bundle projection, zero, fiberwise addition, vertical lift, and canonical flip,
and specified pullbacks satisfying the tangent-category axioms
\citep{rosicky1984abstract,cockett2014differential}.  In compressed notation,
the principal structure maps are
\[
  p:T\Rightarrow\mathrm{Id},
  \qquad
  0:\mathrm{Id}\Rightarrow T,
  \qquad
  +:T\times_{\mathcal X}T\Rightarrow T,
  \qquad
  \ell:T\Rightarrow T^2,
  \qquad
  c:T^2\Rightarrow T^2.
\]
The category of finite-dimensional smooth manifolds, with its ordinary tangent
bundle, is the motivating example.

This language clarifies what it would mean to elevate the compatibility tests
to a categorical differential calculus: copying, discarding, tensor products,
and stochastic composition would have to interact coherently with \(T\) and
its structure maps.  Establishing those axioms for a category of statistical
models or Markov kernels is a substantive construction.

The present paper uses a more concrete and sufficient form of tangent
semantics:
\begin{enumerate}[leftmargin=*,itemsep=2pt]
  \item a regular statistical family has a smooth parameter manifold
        \(\M\), whose ordinary tangent bundle contains intervention fields;
  \item the score map embeds each \(T_\theta\M\) into the finite-dimensional
        score span in \(L_0^2(P_\theta)\); and
  \item a differentiable path of finite-state kernels is linearized in the
        ambient vector space of signed kernels.
\end{enumerate}
These constructions are instances of familiar differentiation, but they do not
by themselves define a tangent functor on the entire Markov category.  This
scoping is important when interpreting the tangent-level analogy with
do-calculus in \cref{sec:tests}.

\section{Basic Statistical Examples}
\label{sec:basic-examples}

The categorical background becomes concrete in regular statistical families.
The following examples show how sufficient statistics, score tangent spaces,
and intervention fields fit together.  They are illustrative model
calculations; causal conclusions require the additional assumptions introduced
later.

\subsection{Gaussian location family}

Let \(\M=\R\) parameterize the unit-variance Gaussian family
\[
  p_\mu(x)
  =
  \frac{1}{\sqrt{2\pi}}
  \exp\!\left[-\frac12(x-\mu)^2\right].
\]
This is a one-dimensional exponential family with sufficient statistic
\(t(x)=x\).  Once observed, the statistic is classical data and may be copied
by \(\Delta(x)=(x,x)\).  The score is
\[
  \partial_\mu\log p_\mu(x)=x-\mu,
\]
so \(T_\mu\M\) is represented by the span of the centered observation
\(x-\mu\).

For a fixed target \(a\), consider the smooth protocol
\[
  \Phi_a(\alpha,\mu)=(1-\alpha)\mu+\alpha a.
\]
Its intervention field is
\[
  v_a(\mu)=(a-\mu)\partial_\mu.
\]
For targets \(a\) and \(b\),
\[
  [v_a,v_b]=(b-a)\partial_\mu.
\]
The bracket is generally nonzero, so the two target-seeking flows do not
commute.  Nevertheless, it remains in the one-dimensional visible tangent
space.  This example separates noncommutativity from nonintegrability: the
visible distribution has rank one and is automatically involutive wherever its
generator is nonzero.

\subsection{Binomial family}

Let \(\M=(0,1)\), let \(X=\{0,1\}^n\), and consider \(n\) independent Bernoulli
trials:
\[
  p_\theta(x)
  =
  \theta^{\sum_i x_i}(1-\theta)^{n-\sum_i x_i}.
\]
The count
\[
  t(x)=\sum_{i=1}^n x_i\in\{0,\ldots,n\}
\]
is sufficient, and its realized value is copied by
\(\Delta(k)=(k,k)\).  The score is
\[
  \partial_\theta\log p_\theta(x)
  =
  \frac{t(x)-n\theta}{\theta(1-\theta)}.
\]
Hence the entire parameter tangent space is generated by the centered count.

A protocol increasing the success probability toward one is
\[
  \Phi_+(\alpha,\theta)=\theta+\alpha(1-\theta),
  \qquad
  v_+(\theta)=(1-\theta)\partial_\theta.
\]
Similarly, a protocol decreasing it toward zero has
\[
  v_-(\theta)=-\theta\partial_\theta.
\]
Their bracket is
\[
  [v_+,v_-]=-\partial_\theta,
\]
which again remains in the one-dimensional tangent space.  The calculation
illustrates that competing interventions can have an order effect even in a
simple independent model.  It does not signal latent structure.

\subsection{Multivariate Gaussian family}

Let
\[
  \M=\R^d\times\mathrm{PD}_d
\]
parameterize \(N(\mu,\Sigma)\).  A sufficient statistic is
\[
  t(x)=(x,xx^\top)\in\R^d\times\mathrm{Sym}_d.
\]
In natural coordinates the corresponding parameters are
\(\Sigma^{-1}\mu\) and \(-\tfrac12\Sigma^{-1}\).  The first- and second-order
statistics generate the score tangent space, and the deterministic statistic
map satisfies
\[
  \Delta\,t=(t\otimes t)\Delta.
\]

Simple mean translations generate commuting fields
\[
  u_i=\partial_{\mu_i},
  \qquad [u_i,u_j]=0.
\]
More general protocols may alter both means and covariances.  For example,
\[
  v_i
  =
  f_i(\mu,\Sigma)\partial_{\mu_i}
  +\sum_{k\leq\ell}
    g^{k\ell}_i(\mu,\Sigma)\partial_{\Sigma_{k\ell}}.
\]
Their brackets are computed from derivatives of the coefficient functions
\(f_i\) and \(g_i^{k\ell}\).  Whether a bracket leaves a proposed visible span
is therefore a property of the specified protocols and span; it is not
determined by Gaussianity or by an off-diagonal covariance alone.

\subsection{Exchangeable data and de Finetti mixtures}

Let \(X_1,\ldots,X_n\) be conditionally i.i.d.\ given a mixing parameter
\(\theta\):
\[
  P(x_1,\ldots,x_n)
  =
  \int_\Theta\prod_{i=1}^n p_\theta(x_i)\,d\mu(\theta).
\]
This representation underlies causal and interventional extensions of de
Finetti theory \citep{guo2022causaldefinetti,guo2024dofinetti}.  The empirical
measure
\[
  \widehat P_n=\frac1n\sum_{i=1}^n\delta_{X_i}
\]
records the unordered sample and is a natural deterministic summary.  In an
exponential-family component model, the corresponding empirical sufficient
statistics generate the score directions for \(\theta\).

Conditional on \(\theta\), a componentwise intervention protocol can be
studied on the parameter manifold of \(p_\theta\).  After mixing over
\(\theta\), its observable derivative also contains the response of the
integral:
\[
  \left.\frac{d}{d\alpha}\right|_0
  \int_\Theta
  p_{\theta,\alpha}(x_1,\ldots,x_n)\,d\mu_\alpha(\theta).
\]
The derivative may include both changes of the component law and changes of
the mixing distribution.  A visible model that represents only the former can
therefore be incomplete.  Whether that incompleteness produces a normal
Lie-bracket residual depends on the chosen intervention fields and visible
span; the mixture representation alone does not determine the answer.

Across these examples, sufficient statistics identify deterministic,
copyable summaries; scores identify tangent directions; and brackets describe
the order dependence of smooth intervention protocols.  The next section
formalizes these intervention fields without requiring that every categorical
object admit a tangent lift.

\section{Statistical Models and Smooth Intervention Protocols}
\label{sec:protocols}

\subsection{Regular statistical models}

Let \(\M\) be a \(q\)-dimensional smooth statistical model
\(\theta\mapsto P_\theta\).  For a dominated regular model with density
\(p_\theta\), the score map sends
\[
  a=\sum_{k=1}^q a^k\partial_k\in T_\theta\M
  \quad\longmapsto\quad
  s_a(x)=\sum_{k=1}^q a^k\partial_k\log p_\theta(x).
\]
If the Fisher information is positive definite, this map is injective and
identifies \(T_\theta\M\) with the finite-dimensional score span
\[
  \mathcal T_\theta
  =
  \Span\{\partial_1\log p_\theta,\ldots,\partial_q\log p_\theta\}
  \subseteq L_0^2(P_\theta),
\]
not with the entire space \(L_0^2(P_\theta)\) of mean-zero square-integrable
functions.  The Fisher metric is
\[
  g_F(a,b)=\mathbb E_{P_\theta}[s_a s_b].
\]
Our differential-geometric results need only a smooth manifold and a chosen
Riemannian metric; the Fisher metric supplies the natural choice in a regular
statistical model \citep{amari2016information}.

\subsection{Intervention protocols}

\begin{definition}[Smooth intervention protocol]
\label{def:protocol}
A smooth local intervention protocol for coordinate or mechanism \(i\) is a map
\[
  \Phi_i:(-\eps,\eps)\times U\longrightarrow\M,
  \qquad
  \Phi_i(0,p)=p,
\]
defined on an open set \(U\subseteq\M\).  Its infinitesimal intervention field
is
\[
  v_i(p)=
  \left.\frac{\partial}{\partial\alpha}\right|_{\alpha=0}
  \Phi_i(\alpha,p)\in T_p\M .
\]
\end{definition}

The parameter \(\alpha\) measures intervention strength around the
observational model.  It is not the value assigned by a hard intervention
\(\doop(X_i=x)\).  A family \(x\mapsto P_{\doop(X_i=x)}\) can generate such a
protocol only after one specifies how the hard-intervention family is connected
smoothly to the observational model.

Different protocols aimed at the same substantive variable can produce
different fields.  Reparameterizing the statistical manifold pushes a fixed
protocol forward, but changing how the intervention is implemented changes the
geometric object itself.

\begin{definition}[Visible intervention distribution]
\label{def:distribution}
For fields \(v_1,\ldots,v_m\), the visible intervention distribution is
\[
  \Dvis(p)=\Span\{v_1(p),\ldots,v_m(p)\}\subseteq T_p\M .
\]
We work on an open stratum on which \(\Dvis\) has constant rank \(r\).  A local
frame \(w_1,\ldots,w_r\) may be selected from smooth combinations of the
generating fields.
\end{definition}

\section{Lie-Bracket Residuals and Integrability}
\label{sec:residuals}

\subsection{Normal bracket residual}

Let \(g\) be a Riemannian metric on \(\M\), and let
\(\Pvis:T\M\to\Dvis\) denote the orthogonal projection on a constant-rank
stratum.

\begin{definition}[Visible bracket residual]
For local sections \(u,v\in\Gamma(\Dvis)\), define
\[
  r(u,v)
  =
  (I-\Pvis)[u,v]\in\Gamma(\Dvis^\perp).
\]
For a chosen frame, write \(r_{ab}=r(w_a,w_b)\).
\end{definition}

The projection makes the residual numerically measurable.  Its \emph{vanishing}
is intrinsic to the distribution, while its norm depends on \(g\).

\begin{proposition}[Residual criterion]
\label{prop:residual-criterion}
Let \(w_1,\ldots,w_r\) be a local frame for a constant-rank distribution
\(\Dvis\).  The following are equivalent:
\begin{enumerate}[label=(\roman*),itemsep=1pt]
  \item \(r_{ab}=0\) for every \(a,b\);
  \item \([w_a,w_b]\in\Gamma(\Dvis)\) for every \(a,b\);
  \item \(\Dvis\) is involutive.
\end{enumerate}
\end{proposition}

\begin{proof}
The first two statements are equivalent by the definition of orthogonal
projection.  If the frame brackets lie in \(\Dvis\), then for
\(u=\sum_a f_aw_a\) and \(v=\sum_b g_bw_b\), the Leibniz rule gives
\[
 [u,v]
 =
 \sum_{a,b}f_ag_b[w_a,w_b]
 +\sum_{a,b}f_aw_a(g_b)w_b
 -\sum_{a,b}g_bw_b(f_a)w_a,
\]
which is again a section of \(\Dvis\).  Thus \(\Dvis\) is involutive.  The
converse follows by applying involutivity to the frame sections.
\end{proof}

\begin{corollary}[Frobenius integrability]
\label{cor:frobenius}
On a constant-rank stratum, all visible bracket residuals vanish if and only if
the visible intervention distribution is locally tangent to a foliation.
\end{corollary}

\begin{proof}
Combine Proposition~\ref{prop:residual-criterion} with the classical Frobenius theorem
\citep{lee2012smooth}.
\end{proof}

This is the precise ``Frobenius'' statement used in this paper.  It is the
differential-geometric Frobenius theorem, not an acyclicity theorem and not a
criterion for causal sufficiency.

\subsection{Invariance and dependence on choices}

\begin{proposition}[Coordinate invariance]
\label{prop:coordinate}
Let \(F:\M\to\mathcal N\) be a diffeomorphism.  Push forward both the
intervention fields and their distribution:
\(\widetilde v_i=F_*v_i\) and
\(\widetilde{\Dvis}=F_*\Dvis\).  Then \(\Dvis\) is involutive if and only if
\(\widetilde{\Dvis}\) is involutive.  Equivalently, the zero-residual property is
preserved when the metric on \(\mathcal N\) is the pushforward metric.
\end{proposition}

\begin{proof}
Naturality of the Lie bracket gives
\([F_*u,F_*v]=F_*[u,v]\).  Hence bracket closure is preserved and reflected by
the invertible differential \(F_*\).
\end{proof}

\begin{proposition}[Frame invariance]
Replacing a local frame \(w_a\) by another smooth local frame of the same
distribution does not change whether every normal residual vanishes.
\end{proposition}

\begin{proof}
The condition is equivalent to involutivity by
Proposition~\ref{prop:residual-criterion}, and involutivity depends only on the
distribution.
\end{proof}

The magnitude \(\|r_{ab}\|_g\), the coefficients of the tangential projection,
and a finite-sample threshold are not invariant under arbitrary changes of
metric, scaling of protocols, or replacement of the visible span.  Those
choices must be part of any empirical specification.

\begin{definition}[Separated stratum]
A constant-rank stratum is \(\gamma\)-separated for a chosen metric, frame, and
protocol normalization if the frame Gram matrix has smallest eigenvalue at
least \(\gamma>0\) and each tested residual satisfies either
\(\|r_{ab}\|_g=0\) or \(\|r_{ab}\|_g\ge\gamma\).
\end{definition}

This is an estimation margin assumption.  It is not implied by regularity of
the statistical model.

\section{Examples and Counterexamples for Bracket Residuals}
\label{sec:examples}

\begin{example}[Commuting coordinate shifts]
On a product location model with parameter \((\mu_1,\ldots,\mu_d)\), translation
protocols generate \(v_i=\partial_{\mu_i}\).  All brackets vanish.  This example
is useful for checking an estimator, but bracket vanishing alone says nothing
about whether the observational variables are confounded.
\end{example}

\begin{example}[Residual without a latent variable]
\label{ex:no-latent-residual}
On the fully observed manifold \(\R^3\) with coordinates \((x,y,z)\), take
\[
  v_1=\partial_x+y\partial_z,
  \qquad
  v_2=\partial_y .
\]
Then
\[
  [v_1,v_2]=-\partial_z,
\]
which is not in \(\Span\{v_1,v_2\}\).  The visible rank-two distribution has a
nonzero normal residual even though no latent variable has been introduced.
The residual records nonintegrability of the chosen protocols, not the cause of
that nonintegrability.
\end{example}

\begin{example}[Latent structure with commuting visible fields]
Let a latent variable \(H\) affect the observational law of visible variables
\(X\) and \(Y\).  Suppose the specified interventions act on a two-dimensional
visible location parameter by translations
\[
  v_X=\partial_{\mu_X},
  \qquad
  v_Y=\partial_{\mu_Y}.
\]
These fields commute regardless of the latent contribution to the baseline
covariance.  Hence a hidden common cause can coexist with zero visible bracket
residual.  Detecting it requires assumptions linking the latent mechanism to
the intervention responses.
\end{example}

The two counterexamples establish that, in the unrestricted framework,
\[
  r_{ij}\ne0
  \quad\not\Longleftrightarrow\quad
  \text{latent confounding}.
\]
A valid detection theorem must specify a structural model class, intervention
completeness, the visible span, and a separation condition under which latent
structure is the only possible source of a normal bracket component.

\section{Copy and Discard in a Linearized Markov Setting}
\label{sec:markov}

\subsection{Why two tangent levels must be separated}

A field \(v_i\in\Gamma(T\M)\) is a tangent vector to a family of probability
laws.  A copy map \(\Delta_X:X\to X\otimes X\) and discard map
\(!_{X}:X\to I\) live in a Markov category.  Expressions such as
\(\Delta_X\circ v_i\) are therefore not automatically typed: one needs a
realization of the parameter tangent as a derivative of morphisms on the sample
object.

We make this explicit in the finite-state case.  Let \(\R^X\) be the vector
space of signed measures on a finite set \(X\).  A Markov kernel is a stochastic
linear map \(K:\R^X\to\R^Y\).  The discard functional
\(!_{X}:\R^X\to\R\) sums total mass, and the copy map is
\[
  \Delta_X e_x=e_{(x,x)}
\]
on basis distributions.

\begin{definition}[Kernel derivative]
Let \(K_\alpha:\R^X\to\R^Y\) be a differentiable path of stochastic kernels
with \(K_0=K\).  Its derivative is the signed linear map
\[
  D=\left.\frac{d}{d\alpha}\right|_{\alpha=0}K_\alpha .
\]
The map \(D\) need not be stochastic.
\end{definition}

\begin{proposition}[Infinitesimal normalization]
\label{prop:normalization}
Every differentiable path of stochastic kernels satisfies
\[
  !_{Y}\circ D=0 .
\]
\end{proposition}

\begin{proof}
For every \(\alpha\), stochasticity gives
\(!_{Y}\circ K_\alpha=!_{X}\).  Differentiate at \(\alpha=0\).
\end{proof}

Thus first-order discard preservation expresses conservation of total
probability mass.  It is automatic for a valid stochastic path and does not
encode deletion of an irrelevant causal action.

\subsection{Copy compatibility}

Assume \(X=Y\), \(K_0=\id\), and consider a path of endomorphisms.  Exact
copy-preservation would mean
\[
  \Delta_XK_\alpha
  =
  (K_\alpha\otimes K_\alpha)\Delta_X .
\]
Differentiating at zero yields the coderivation identity
\[
  \Delta_XD
  =
  (D\otimes\id+\id\otimes D)\Delta_X .
\]

\begin{definition}[Copy defect]
The first-order copy defect of \(D\) is
\[
  B_D
  =
  \Delta_XD
  -(D\otimes\id+\id\otimes D)\Delta_X .
\]
\end{definition}

\begin{proposition}[Linearized copy criterion]
\label{prop:copy}
The path \(K_\alpha\) preserves copying to first order at zero if and only if
\(B_D=0\).
\end{proposition}

\begin{proof}
Differentiate the exact copy-preservation equation at \(\alpha=0\).
\end{proof}

This criterion is deliberately modest.  In a Markov category every object
carries copy and discard as a commutative comonoid
\citep{fritz2020synthetic}.  A full Frobenius algebra additionally requires
multiplication and unit maps satisfying further laws; it should not be inferred
from copy/discard alone.  Moreover, generic stochastic kernels do not preserve
copying: copy-preserving kernels are deterministic under standard hypotheses.
The defect \(B_D\) therefore measures first-order failure of deterministic
information flow; by itself, it establishes neither conditional independence
nor action--observation exchange.

\subsection{Compatibility with parameter-space fields}

To connect \(v_i\) from Definition~\ref{def:protocol} to a kernel derivative \(D_i\), one
must specify a differentiable realization
\[
  \rho_i:
  \text{intervention protocol on }\M
  \longrightarrow
  \text{path of kernels on }X .
\]
Only after this choice can one compare the geometric residuals
\((I-\Pvis)[v_i,v_j]\) with the linear defects \(B_{D_i}\).  Such a realization
may exist for a structural equation model or a controlled Markov process, but
it is not supplied by the tangent manifold alone.

\section{Three \texorpdfstring{\IC}{IC} Compatibility Tests}
\label{sec:tests}

The preceding constructions lead to three \IC\ compatibility tests:
\[
\begin{array}{rcll}
\text{normalization}      &:& !\circ D_i=0,\\[1mm]
\text{copy compatibility}&:& B_{D_i}=0,\\[1mm]
\text{visible involutivity}&:&
  (I-\Pvis)[v_i,v_j]=0\quad\text{for all }i,j.
\end{array}
\]
They answer different questions.  The first asks whether an estimated
first-order perturbation remains inside normalized probability laws.  The
second asks whether a kernel perturbation preserves deterministic copying to
first order.  The third asks whether successive local intervention flows remain
inside the distribution generated by the modeled interventions.  This section
develops each test and then explains the precise sense in which the first two
suggest tangent-level analogies with operations in do-calculus.

\begin{figure}[H]
\centering
\begin{tikzpicture}[
  x=1cm,y=1cm,
  wire/.style={-{Latex[length=2mm]},semithick},
  plainwire/.style={semithick},
  op/.style={draw,rounded corners=1pt,minimum width=0.9cm,
             minimum height=0.55cm,fill=white},
  smallop/.style={draw,rounded corners=1pt,minimum width=0.65cm,
                  minimum height=0.48cm,fill=white},
  every node/.style={font=\small}
]
% Panel (a): localized marginal invariance
\node[anchor=east,font=\small\bfseries] at (0.2,4.8) {(a)};
\coordinate (a0) at (0.7,4.8);
\node[op] (aD) at (2.0,4.8) {\(D_i\)};
\node[op] (am) at (3.7,4.8) {\(m\)};
\node (az) at (5.0,4.8) {\(0\)};
\draw[wire] (a0) -- (aD);
\draw[wire] (aD) -- (am);
\draw[wire] (am) -- (az);
\node[anchor=west] at (5.7,4.8)
  {\(mD_i=0\quad\) (take \(m={!_Y}\) for normalization)};

% Panel (b): coderivation/copy compatibility
\node[anchor=east,font=\small\bfseries] at (0.2,2.6) {(b)};
\coordinate (b0) at (0.7,2.6);
\node[smallop] (bD) at (1.7,2.6) {\(D_i\)};
\fill (2.75,2.6) circle (1.8pt);
\draw[wire] (b0) -- (bD);
\draw[plainwire] (bD) -- (2.75,2.6);
\draw[wire] (2.75,2.6) -- (3.65,3.0);
\draw[wire] (2.75,2.6) -- (3.65,2.2);
\node at (4.25,2.6) {\(=\)};

\coordinate (c10) at (4.75,2.6);
\fill (5.55,2.6) circle (1.8pt);
\node[smallop] (c1D) at (6.65,3.0) {\(D_i\)};
\draw[plainwire] (c10) -- (5.55,2.6);
\draw[wire] (5.55,2.6) -- (c1D);
\draw[wire] (5.55,2.6) -- (7.55,2.2);
\draw[wire] (c1D) -- (7.55,3.0);
\node at (8.05,2.6) {\(+\)};

\coordinate (c20) at (8.55,2.6);
\fill (9.35,2.6) circle (1.8pt);
\node[smallop] (c2D) at (10.45,2.2) {\(D_i\)};
\draw[plainwire] (c20) -- (9.35,2.6);
\draw[wire] (9.35,2.6) -- (11.35,3.0);
\draw[wire] (9.35,2.6) -- (c2D);
\draw[wire] (c2D) -- (11.35,2.2);
\node[anchor=west] at (11.65,2.6) {\(B_{D_i}=0\)};

% Panel (c): normal bracket closure
\node[anchor=east,font=\small\bfseries] at (0.2,0.4) {(c)};
\coordinate (d0) at (0.7,0.4);
\node[op,minimum width=1.25cm] (bracket) at (2.0,0.4)
  {\([v_i,v_j]\)};
\node[op,minimum width=1.5cm] (proj) at (4.2,0.4)
  {\(I-\Pi_{\rm vis}\)};
\node (dz) at (5.65,0.4) {\(0\)};
\draw[wire] (d0) -- (bracket);
\draw[wire] (bracket) -- (proj);
\draw[wire] (proj) -- (dz);
\node[anchor=west] at (6.3,0.4)
  {closure of the sequential-intervention commutator};
\end{tikzpicture}
\caption{\textbf{String-diagram summary of the \IC\ compatibility
tests.}  (a) A selected marginal or readout annihilates the kernel derivative;
ordinary normalization is the special case \(m={!_Y}\).
(b) Perturbing and then copying equals the sum of copying first and perturbing
each output leg, which is the coderivation identity \(B_{D_i}=0\).
(c) The Lie bracket of two intervention fields has no component normal to the
visible distribution.  The diagrams express tangent-level compatibility; a
connection to finite do-calculus operations additionally requires the bridge
conditions described in \cref{sec:tests}.}
\label{fig:compatibility-diagrams}
\end{figure}
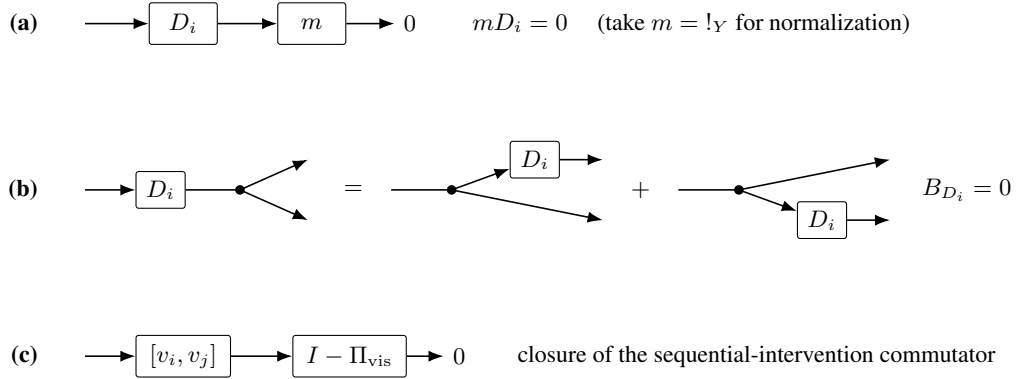

\subsection{Normalization and localized marginal invariance}

Let \(K_\alpha:X\rightsquigarrow Y\) be a differentiable path of stochastic
kernels with derivative \(D_i\) at \(\alpha=0\).  Since
\(!_{Y}\circ K_\alpha=!_{X}\) for every \(\alpha\), differentiation gives
\[
  !_{Y}\circ D_i=0.
\]
Thus the columns of \(D_i\) have total mass zero.  This is the kernel analogue
of the familiar fact that a score has expectation zero: an infinitesimal
probability perturbation may redistribute mass but cannot create or destroy
it.  In applications this identity is a basic validity check.  If an estimated
\(\widehat D_i\) has \(!_{Y}\widehat D_i\neq0\), it is not tangent to the
space of stochastic kernels, irrespective of its causal interpretation.

Global normalization should be distinguished from \emph{localized
irrelevance}.  Let
\[
  m:Y\rightsquigarrow Z
\]
be a specified marginalization or readout kernel.  The observable represented
by \(m\) is unchanged to first order along the intervention path precisely when
\[
  m\circ D_i=0.
\]
Indeed,
\[
  m\circ K_\alpha
  =
  m\circ K_0+\alpha(m\circ D_i)+o(\alpha),
\]
so the condition says that applying the intervention and then retaining the
\(Z\)-readout agrees, to first order, with the unperturbed readout.

\begin{proposition}[First-order marginal invariance]
\label{prop:marginal-invariance}
For a differentiable kernel path
\(K_\alpha=K_0+\alpha D_i+o(\alpha)\) and a fixed readout \(m\),
\[
  m\circ K_\alpha=m\circ K_0+o(\alpha)
  \quad\Longleftrightarrow\quad
  m\circ D_i=0.
\]
\end{proposition}

\begin{proof}
Compose the first-order expansion of \(K_\alpha\) with the linear map \(m\)
and compare the coefficient of \(\alpha\).
\end{proof}

This localized condition, rather than normalization alone, carries an
irrelevance interpretation.  For example, if \(Y=Y_1\times Y_2\) and \(m\)
marginalizes out \(Y_2\), then \(mD_i=0\) says that the intervention changes
the joint law only in directions invisible to the \(Y_1\)-marginal.  It makes
no claim about conditional independence: that conclusion would require a
causal model and an adjustment or separation condition specifying why this
marginal invariance holds.

\subsection{Copy compatibility and deterministic information}

For an endomorphism path based at the identity, the copy defect from
Proposition~\ref{prop:copy} is
\[
  B_{D_i}
  =
  \Delta_XD_i
  -(D_i\otimes\id+\id\otimes D_i)\Delta_X .
\]
The two terms describe different orders of operation.  In
\(\Delta_XD_i\), the system is perturbed and its output is then copied.  In
\((D_i\otimes\id+\id\otimes D_i)\Delta_X\), the unperturbed value is copied
first and the perturbation is applied in either output leg.  Consequently,
\(B_{D_i}=0\) is a Leibniz or coderivation identity: copying commutes with the
perturbation to first order.

This is much more restrictive than normalization.  A generic stochastic
perturbation introduces fresh randomness.  Copying one random outcome twice is
not the same as drawing two conditionally independent outcomes after copying
the input.  Copy preservation therefore detects deterministic information
flow, not arbitrary probabilistic evolution.  In a finite state space,
deterministic kernels form a discrete subset of the stochastic polytope, so a
nontrivial smooth stochastic path will generally have a nonzero copy defect.
The value of the test is precisely to expose that distinction.

For a general base kernel \(K_0:X\rightsquigarrow Y\), differentiating
\[
  \Delta_YK_\alpha
  =
  (K_\alpha\otimes K_\alpha)\Delta_X
\]
gives the base-point-dependent defect
\[
  B_{D_i;K_0}
  :=
  \Delta_YD_i
  -(D_i\otimes K_0+K_0\otimes D_i)\Delta_X.
\]
Thus \(B_{D_i;K_0}=0\) means that a copy-preserving realization remains
copy-preserving to first order at \(K_0\).  The earlier expression \(B_{D_i}\)
is the special case \(X=Y\) and \(K_0=\id\).

The causal use of this test requires an additional comparison.  Suppose an
observational update and an interventional update are represented by two
differentiable paths \(K_\alpha^{\mathrm{obs}}\) and
\(K_\alpha^{\mathrm{act}}\), after whatever adjustment the causal model
requires.  A first-order action--observation exchange statement would have to
establish
\[
  \left.\frac{d}{d\alpha}\right|_0
  K_\alpha^{\mathrm{obs}}
  =
  \left.\frac{d}{d\alpha}\right|_0
  K_\alpha^{\mathrm{act}}
\]
on the adjusted readout, together with the relevant copy and normalization
compatibilities.  Vanishing copy defect makes such a comparison structurally
coherent; it does not by itself prove the equality or supply the adjustment
condition.

\subsection{Visible involutivity and closure under sequential intervention}

The third test belongs to the geometry of the parameter manifold.  Let
\(\Phi_i^t\) and \(\Phi_j^t\) be local flows generated by intervention fields
\(v_i\) and \(v_j\).  The Lie bracket measures the leading discrepancy between
performing the two small interventions in opposite orders.  With a consistent
choice of commutator convention,
\[
  \Phi_j^{-t}\circ\Phi_i^{-t}\circ\Phi_j^t\circ\Phi_i^t(p)
  =
  p+t^2[v_i,v_j](p)+O(t^3).
\]
Projecting this displacement normal to \(\Dvis\) yields
\[
  t^2(I-\Pvis)[v_i,v_j](p)+O(t^3).
\]
A zero residual therefore means that the order effect can still be expressed,
to leading nontrivial order, using the modeled visible intervention directions.
A nonzero residual means that closing the small commutator loop generates a
direction missing from the chosen visible span.

Requiring zero residuals for every pair of local generators makes \(\Dvis\)
involutive.  By Proposition~\ref{prop:residual-criterion}, this condition extends from a
frame to all local sections of the distribution; by the Frobenius theorem, the
distribution is then tangent to local intervention leaves.  Each leaf collects
model points reachable by composing the visible local flows, at least within
the neighborhood in which the constant-rank assumptions hold.

The test concerns closure, not commutativity.  A bracket may be nonzero and
still lie entirely in \(\Dvis\).  In that case intervention order matters, but
the order effect remains representable within the modeled intervention family.
Conversely, the residual depends on which directions were included in
\(\Dvis\): omitting a legitimate direction can create a residual, while adding
a sufficiently flexible direction can absorb one.

\subsection{Relation to Pearl's rules}

Pearl's three rules concern equality of finite interventional probability
expressions under conditional-independence statements read from mutilated
causal graphs \citep{pearl2009causality}.  The \IC\ compatibility tests above
concern derivatives and local flows.  Nevertheless, there is a useful analogy
at the level of the operations being controlled:
\begin{center}
\begin{tabular}{p{0.25\linewidth}p{0.30\linewidth}p{0.34\linewidth}}
\toprule
\textbf{Classical operation} & \textbf{Tangent-level condition}
& \textbf{Additional bridge required}\\
\midrule
Insert or delete an irrelevant observation or action
& \(mD_i=0\): the selected marginal or readout is invariant to first order
& A causal graph or structural model establishing the relevant conditional
  independence, plus invariance along the finite flow\\
\addlinespace
Exchange an observation with an action
& Equality of adjusted observational and interventional derivatives, with
  normalization and copy coherence
& An explicit observation--action comparison map and the corresponding
  adjustment assumptions\\
\addlinespace
Compose admissible local interventions
& \((I-\Pvis)[v_i,v_j]=0\): the visible family is closed under infinitesimal
  commutators
& Completeness of the intervention fields and integration from local flows to
  the finite operations of interest\\
\bottomrule
\end{tabular}
\end{center}

The first row can play the tangent role of either deletion of observations or
deletion of actions, depending on what \(D_i\) represents and which readout
\(m\) is retained.  The second row is the closest analogue of
action--observation exchange, but its central content is equality of two
derivatives after adjustment; copy compatibility is a coherence condition on
that equality.  The third row has no exact one-to-one counterpart among
Pearl's rules.  It is an additional geometric requirement ensuring that the
chosen family of infinitesimal interventions is closed under sequential
composition.

This is the sense in which the tests form a \emph{tangent-level shadow} of
some operations appearing in do-calculus.  The phrase does not require a
tangent structure on the entire category of stochastic kernels.  In the
finite-state construction, differentiation takes place in the ambient
finite-dimensional vector spaces of signed kernels; on \(\M\), it takes place
in the ordinary tangent bundle.  A genuinely tangent-categorical formulation
would additionally have to specify a tangent functor, its natural
transformations and coherence axioms, and its compatibility with the monoidal
and Markov structure.

\subsection{When the \texorpdfstring{\IC}{IC} tests form a joint calculus}

The \IC\ tests live at two mathematical levels.  Normalization, marginal invariance,
and copy compatibility concern derivatives \(D_i\) of kernels.  Visible
involutivity concerns vector fields \(v_i\) on a statistical model.  To use them
together, one must specify a differentiable realization
\[
  \rho_i:v_i\longmapsto D_i
\]
that describes how a parameter-space intervention field changes the relevant
stochastic kernel.  The realization should also identify the readouts \(m\),
the visible distribution, and any observational--interventional comparison
maps.

Once these data are supplied, the three tests provide a layered diagnostic:
\begin{enumerate}[leftmargin=*,itemsep=3pt]
  \item normalization verifies that the realized derivative is probabilistically
        admissible;
  \item marginal and copy tests determine which readouts and deterministic
        information structures it preserves to first order; and
  \item involutivity determines whether the realized family is geometrically
        closed under local sequential intervention.
\end{enumerate}
Passing all three establishes local compatibility of the specified
realization.  Promoting that compatibility to a causal identification result
requires the structural and statistical assumptions developed in
\cref{sec:identification}.

\section{Relation of \texorpdfstring{\IC}{IC} to Estimation and BRIDGE/SKFM}
\label{sec:estimation}

The geometric portion of \IC\ suggests an empirical workflow:
\begin{enumerate}[leftmargin=*,itemsep=2pt]
  \item specify or estimate intervention-response fields \(v_i\);
  \item choose a visible distribution and metric;
  \item estimate Jacobians and brackets
        \([v_i,v_j]=Dv_j\,v_i-Dv_i\,v_j\);
  \item project brackets onto the estimated visible span; and
  \item report tangential coefficients, normal residuals, conditioning of the
        frame Gram matrix, and uncertainty under resampling.
\end{enumerate}

BRIDGE/SKFM is an algorithmic realization of this pattern
\citep{mahadevan2026latentconfoundedcausaldiscovery}.  Within \IC, its bracket
and center statistics are diagnostics of the estimated intervention geometry.
Interpreting them as detectors of latent confounding additionally requires an
identification theorem for the algorithm's structural model class.

Several practical cautions follow:
\begin{itemize}[leftmargin=*,itemsep=2pt]
  \item numerical differentiation amplifies field-estimation error;
  \item projection residuals are unstable when the visible frame is nearly
        rank deficient;
  \item a change in intervention scaling changes bracket magnitudes;
  \item omission of a relevant visible direction can manufacture a residual;
        and
  \item adding flexible directions can absorb a residual without explaining
        its causal source.
\end{itemize}

A finite-sample result should therefore control field estimation, Jacobian
estimation, subspace perturbation, and the gap separating zero from nonzero
normal components.  These statistical questions are outside the purely
differential-geometric Frobenius theorem.

\section{What Is Needed for a Causal Identification Theorem}
\label{sec:identification}

To turn the residual diagnostic into a statement about hidden causal structure,
one needs assumptions beyond smoothness:
\begin{enumerate}[leftmargin=*,itemsep=2pt]
  \item a structural causal model or controlled-kernel class specifying how
        interventions alter mechanisms;
  \item intervention validity and a known relation between experimental controls
        and the fields \(v_i\);
  \item a completeness condition ensuring that the visible span contains all
        directions generated in the causally sufficient case;
  \item an exclusion condition ruling out fully observed nonintegrable protocols
        such as Example~\ref{ex:no-latent-residual};
  \item a faithfulness or separation condition preventing hidden structure from
        producing a zero residual; and
  \item an estimation theorem transferring population residual gaps to finite
        samples.
\end{enumerate}

Under such assumptions one may seek implications of the form
\[
 \text{normal residual}
 \Longrightarrow
 \text{model misspecification or omitted causal mechanism},
\]
followed by additional conditions that isolate latent confounding from other
forms of misspecification.  The unrestricted geometric framework alone cannot
make that final distinction.

Likewise, recovering finite Pearl do-calculus identities from infinitesimal
conditions would require substantially more than a derivative vanishing at one
point: one needs invariance along an entire flow, existence and completeness of
that flow, compatible graph-surgery semantics, and the appropriate chartwise
conditional-independence premises.  No such recovery theorem is asserted here.

\section{Discussion}

Here ``Frobenius'' has its classical differential-geometric meaning: the
integrability theorem for
constant-rank distributions.  Markov-category copy/discard structure is treated
separately as a commutative comonoid, and its tangent calculations are performed
only in a linearized setting where derivatives are typed.

This separation clarifies the contribution of \IC.  Lie brackets measure the order
dependence of infinitesimal intervention protocols.  Normal residuals measure
failure of the chosen visible distribution to close.  Copy defects measure
failure of a chosen kernel perturbation to preserve deterministic copying.
These are related diagnostics only when a concrete causal model supplies maps
between the parameter, kernel, and observation levels.

Graphs remain valuable when a structural model is available, but the geometric
diagnostics can be computed before selecting a unique graph presentation.  That
presentation independence is pragmatic rather than absolute: the fields,
visible span, metric, and intervention implementation must still be specified.

\section{Changes in This Revision}
\label{sec:changes}

This revision makes the following substantive corrections:
\begin{enumerate}[leftmargin=*,itemsep=2pt]
  \item removes the claimed Kan-extension duality between conditioning and
        intervention and deletes tangent Kan-transport results derived from it;
  \item recasts the three infinitesimal conditions as compatibility tests and
        states precisely the additional assumptions needed for their
        tangent-level analogy with operations in do-calculus;
  \item replaces the asserted recovery of discrete do-calculus with an explicit
        statement of the additional hypotheses such a theorem would require;
  \item corrects the statistical tangent space from all of \(L_0^2(P)\) to the
        finite-dimensional score span;
  \item retains \(\StatInf\) as a category of presented statistical models but
        replaces its unverified tangent endofunctor with the explicit
        forgetful functor to smooth parameter manifolds;
  \item separates parameter-space vector fields from derivatives of stochastic
        kernels and supplies a well-typed finite-state linearization;
  \item distinguishes copy/discard comonoids from full Frobenius algebras;
  \item removes the undeveloped Hochschild obstruction class and categorical
        presentation adjunction;
  \item states the valid residual--involutivity equivalence as an application of
        the classical Frobenius theorem; and
  \item adds counterexamples showing that bracket residuals do not characterize
        latent confounding without structural assumptions.
\end{enumerate}

\section{Conclusion}

The central idea of \IC\ is simple: a smooth intervention
protocol has a derivative, and several such derivatives define a local
distribution of intervention directions.  Lie brackets reveal whether that
distribution is closed, while normal residuals quantify its failure to be
integrable.  This geometry is rigorous, coordinate invariant at the level of
vanishing, and directly relevant to algorithms that estimate local intervention
fields.

Its causal meaning is conditional rather than automatic.  Residuals diagnose a
chosen intervention geometry; they do not by themselves identify latent
confounding, recover a causal graph, or reproduce do-calculus.  The geometric
theorem and the accompanying compatibility diagnostics provide a foundation on
which model-specific identification and finite-sample results can be built.

\bibliography{references}

@misc{mahadevan2026latentconfoundedcausaldiscovery,
      title={{Latent Confounded Causal Discovery via Lie Bracket Geometry}}, 
      author={Sridhar Mahadevan},
      year={2026},
      eprint={2606.19610},
      archivePrefix={arXiv},
      primaryClass={cs.LG},
      url={https://arxiv.org/abs/2606.19610}, 
}

@book{pearl2009causality,
  title     = {Causality: Models, Reasoning, and Inference},
  author    = {Pearl, Judea},
  edition   = {2},
  year      = {2009},
  publisher = {Cambridge University Press}
}

@article{rubin2005causal,
  title   = {Causal Inference Using Potential Outcomes: Design, Modeling, Decisions},
  author  = {Rubin, Donald B.},
  journal = {Journal of the American Statistical Association},
  volume  = {100},
  number  = {469},
  pages   = {322--331},
  year    = {2005},
  doi     = {10.1198/016214504000001880}
}

@book{lee2012smooth,
  title     = {Introduction to Smooth Manifolds},
  author    = {Lee, John M.},
  series    = {Graduate Texts in Mathematics},
  volume    = {218},
  edition   = {2},
  year      = {2012},
  publisher = {Springer},
  doi       = {10.1007/978-1-4419-9982-5}
}

@book{amari2016information,
  title     = {Information Geometry and Its Applications},
  author    = {Amari, Shun-ichi},
  year      = {2016},
  publisher = {Springer},
  doi       = {10.1007/978-4-431-55978-8}
}

@misc{guo2022causaldefinetti,
  title         = {Causal de {Finetti}: On the Identification of Invariant Causal Structure in Exchangeable Data},
  author        = {Guo, Siyuan and T{\'o}th, Viktor and Sch{\"o}lkopf, Bernhard and Husz{\'a}r, Ferenc},
  year          = {2022},
  eprint        = {2203.15756},
  archivePrefix = {arXiv},
  primaryClass  = {cs.LG},
  url           = {https://arxiv.org/abs/2203.15756}
}

@misc{guo2024dofinetti,
  title         = {Do {Finetti}: On Causal Effects for Exchangeable Data},
  author        = {Guo, Siyuan and Zhang, Chi and Mohan, Karthika and Husz{\'a}r, Ferenc and Sch{\"o}lkopf, Bernhard},
  year          = {2024},
  eprint        = {2405.18836},
  archivePrefix = {arXiv},
  primaryClass  = {cs.LG},
  url           = {https://arxiv.org/abs/2405.18836}
}

@article{fritz2020synthetic,
  title   = {A Synthetic Approach to {Markov} Kernels, Conditional Independence and Theorems on Sufficient Statistics},
  author  = {Fritz, Tobias},
  journal = {Advances in Mathematics},
  volume  = {370},
  pages   = {107239},
  year    = {2020},
  doi     = {10.1016/j.aim.2020.107239},
  url     = {https://arxiv.org/abs/1908.07021}
}

@article{fritz2023dseparation,
  title   = {The {d-Separation} Criterion in Categorical Probability},
  author  = {Fritz, Tobias and Klingler, Andreas},
  journal = {Journal of Machine Learning Research},
  volume  = {24},
  number  = {46},
  pages   = {1--49},
  year    = {2023},
  url     = {https://jmlr.org/papers/v24/22-0916.html}
}

@misc{jacobs2018causal,
  title         = {Causal Inference by String Diagram Surgery},
  author        = {Jacobs, Bart and Kissinger, Aleks and Zanasi, Fabio},
  year          = {2018},
  eprint        = {1811.08338},
  archivePrefix = {arXiv},
  primaryClass  = {cs.LO},
  url           = {https://arxiv.org/abs/1811.08338}
}

@article{cho2017disintegration,
  title         = {Disintegration and {Bayesian} Inversion via String Diagrams},
  author        = {Cho, Kenta and Jacobs, Bart},
  journal       = {arXiv preprint arXiv:1709.00322},
  year          = {2017},
  eprint        = {1709.00322},
  archivePrefix = {arXiv},
  primaryClass  = {math.CT},
  url           = {https://arxiv.org/abs/1709.00322}
}

@article{rosicky1984abstract,
  title   = {Abstract Tangent Functors},
  author  = {Rosick{\'y}, Ji{\v r}{\'i}},
  journal = {Diagrammes},
  volume  = {12},
  pages   = {JR1--JR11},
  year    = {1984}
}

@article{cockett2014differential,
  title   = {Differential Structure, Tangent Structure, and {SDG}},
  author  = {Cockett, J. R. B. and Cruttwell, G. S. H.},
  journal = {Applied Categorical Structures},
  volume  = {22},
  pages   = {331--417},
  year    = {2014},
  doi     = {10.1007/s10485-013-9312-0}
}
\bibliographystyle{abbrvnat}

\end{document}